\documentclass[a4paper,12pt]{amsart}
\usepackage[top=3cm,bottom=3cm,outer=3cm,inner=3cm,marginpar=3cm]{geometry}

\newtheorem{theorem}{Theorem}[section]
\newtheorem{lemma}[theorem]{Lemma}

\newtheorem{proposition}[theorem]{Proposition}

\theoremstyle{definition}
\newtheorem{definition}{Definition}

\newtheorem{remark}[theorem]{Remark}

\numberwithin{equation}{section}

\def\CC{\mathbb{C}} 
\def\ric{\mathrm{Ric}}
\def\ke{K\"ahler-Einstein }

\newcommand{\paren}[1]{\left(#1\right)}

\newcommand{\abs}[1]{\left\vert#1\right\vert}
\newcommand{\norm}[1]{\left\|#1\right\|}

\newcommand{\ov}[1]{\overline{#1}}

\begin{document}
\title[Holomophic family of strongly pseudoconvex domains]{Holomophic family of strongly pseudoconvex domains in a K\"{a}hler manifold}

\author{Young-Jun Choi}
\address{Department of Mathematics, Pusan National University, 2, Busandaehak-ro 63beon-gil, Geumjeong-gu, Busan 46241, Republic of Korea}
\email{youngjun.choi@pusan.ac.kr}

\author{Sungmin Yoo}
\address{School of Mathematics, Korea Institute for Advanced Study(KIAS), 85 Hoegiro, Dongdaemun-gu, Seoul 02455, Republic of Korea}
\email{sungmin@kias.re.kr}

\begin{abstract}
Let $p:X\rightarrow Y$ be a surjective holomorphic mapping between K\"ahler manifolds. 
Let $D$ be a bounded smooth domain in $X$ such that every generic fiber $D_y:=D\cap p^{-1}(y)$ for $y\in Y$ is a strongly pseudoconvex domain in $X_y:=p^{-1}(y)$, which admits the complete K\"ahler-Einstein metric.
This family of K\"ahler-Einstein metrics induces a smooth $(1,1)$-form $\rho$ on $D$.
In this paper, we prove that $\rho$ is positive-definite on $D$ if $D$ is strongly pseudoconvex.
We also discuss the extensioin of $\rho$ as a positive current across singular fibers.
\end{abstract}

\maketitle

\section{Introduction}

Let $p:X\rightarrow Y$ be a surjective holomorphic map, where $X$ and $Y$ are complex manifolds and let $D$ be a bounded smooth domain in $X$.
We denote by $W\subset Y$ the set of all singular values of $p\vert_D$ and $p\vert_{\partial D}$ in $S:=p(D)$.
If every generic fiber $D_y:=D\cap p^{-1}(y)$ with $y\in S':=S\setminus W$ is a bounded strongly pseudoconvex domain in $X_y:=p^{-1}(y)$ and $p$ is proper on $\overline D$, then we call $p:D\rightarrow S$ a \emph{holomorphic family of bounded strongly pseudoconvex domains} in $X$
(with degenerations).

If there exists a K\"{a}hler form $\omega_X$ on $X$ and satisfies that the Ricci curvature $\ric(\omega_y)$ of $\omega_y:=\omega_X\vert_{X_y}$ is negatively curved for every $y\in S'$, then Cheng and Yau's theorem implies that there exists a unique complete K\"ahler metric $\omega^{KE}_y$ on $D_y$ satisfying
$$
\ric(\omega^{KE}_y)=-(n+1)\omega^{KE}_y,
$$
where $n$ is the dimension of $D_y$ (cf. \cite{Cheng_Yau,Coevering}). This metric $\omega^{KE}_y$ is called the \emph{K\"ahler-Einstein metric with Ricci curvature $-(n+1)$}.
Since the map $p$ is proper on $\overline D$, this family of K\"ahler-Einstein metrics induces a smooth hermitian metric $h_{D'/S'}$ on the relative canonical line bundle $K_{D'/S'}$, where $D':=D\setminus p^{-1}(W)$. The corresponding curvature is defined by
$$
\Theta_{h_{D'/S'}}
:=
i\partial\ov{\partial}\log({(\omega_y^{KE})}^n\wedge p^{\ast}(dV_s)),
$$
where $dV_s$ is the Euclidean volume form in any local holomorphic coordinates of $y\in S$ (for detail, see Section \ref{S:construction}).
The K\"{a}hler-Einstein condition implies that
$$
\Theta_{h_{D'/S'}}|_{D_y}
=
i\partial\ov{\partial}\log({\omega^{KE}_y})^n
=
-\ric(\omega^{KE}_y)=(n+1)\omega^{KE}_y,
$$
for all $y\in S'$. 
Hence $\Theta_{h_{D'/S'}}$ is a $d$-closed real $(1,1)$-form on $D'$ which is the K\"ahler-Einstein metric on each fiber $D_y$.
This is called a \emph{variation of \ke metric} or a \emph{fiberwise \ke metric}, which will be denoted by $\rho$ in Section \ref{S:construction}.
\medskip

The first theorem is about the positivity of $\rho$ in the total space $D'$.

\begin{theorem}\label{main}
Let $p:D\rightarrow S$ be a holomorphic family of strongly pseudoconvex domains in $X$. 
Under the above assumption, if the total space $D$ is strongly pseudoconvex, then $\rho$ is positive-definite on $D'$.
\end{theorem}

Since $\rho$ is a $d$-closed positive smooth $(1,1)$-form on $D'$, it is natural to consider the extension of $\rho$ across the singular fibers $p^{-1}(W)$.
The next theorem gives the answer.

\begin{theorem}\label{extension}
Suppose that $W$ is contained in an analytic subset of $S$. 
If $D$ admits a complete K\"{a}hler metric $\widetilde\omega_D$ such that the scalar curvature of $\widetilde\omega_y:=\widetilde\omega_D\vert_{D_y}$ for $y\in S'$ is uniformly bounded from below, then $\rho$ extends to $D$ as a positive current.
\end{theorem}

Variations of \ke metrics have been studied by many authors (see for instance, \cite{Siu, Schumacher, Paun, Choi1, Choi2, Cao_Guenancia_Paun}, et al.). 
In 2012, Schumacher proved that the variation of \ke metrics for a family of canonically polarized compact K\"ahler manifolds is positive (\cite{Schumacher}).
Later, P\u aun generalized this result to twisted case and extended the variation across the singularities using the Ohsawa-Takegoshi Theorem (\cite{Paun}). 
In case of complete manifolds case, the first author proved the positivity of the variation for a family of pseudoconvex domains in the complex Euclidean space, when $p$ is the coordinate projection.

In this paper, we consider a family of pseudoconvex domains in not only the complex Euclidean space but also general K\"ahler manifolds, when $p$ is an arbitrary surjective map.
As previous results, our variation of \ke metrics satisfies Schumacher's PDE.
Then a careful estimate of the boundary behavior of the geodesic curvature (which is a invariant encoding the positivity of a variation) completes the proof of Theorem \ref{main}. This is obtained by combining results and techniques in \cite{Choi1, Choi2} and \cite{Coevering}.

For the extension of the variation $\rho$ across $D\setminus D'$, we will follow the lines in \cite{Paun}.
The main difference is a lack of uniform boundedness of fiber volumes.
In the previous case, since the fibers are compact and the total space is K\"ahler, every fiber has the same volume.
But in our case, a priori fiber does not have finite volume since the fiber is noncompact. 
Moreover, the local uniform boundedness is not easily obtained.
This is why the existence of the metric $\widetilde\omega_D$ is necessary in the hypothesis of Theorem \ref{extension}.
\bigskip

\noindent\textbf{Acknowledgement.}
The authors would like to thank to G. Schumacher and M. P\u aun for their valuable comments and suggestions. 
The first author was supported by the National Research Foundation
(NRF) of Korea grant funded by the Korea government (No. 2018R1C1B3005963).
The second author was supported by the National Research Foundation
(NRF) of Korea grant funded by the Korea government (No. 2010-0020413).

\section{Preliminaries}

In this section, we briefly review the results due to Cheng and Yau: the Monge-Amp\`{e}re equation, the construction of the K\"{a}hler-Einstein metric on a strongly pseudoconvex domain in a K\"{a}hler manifold, and its boundary behavior. For more details, we refer to \cite{Cheng_Yau,Coevering}.

\subsection{K\"{a}hler-Einstein metric on a strongly pseudoconvex domain}

Let $\Omega$ be a smooth bounded strongly pseudoconvex domain in a K\"{a}hler manifold $(M,\omega)$ satisfying 
$
\ric(\omega)<0.
$
This gives us a new K\"{a}hler form $\omega^0$ on $M$, defined by
$$
\omega^0:=-\frac{1}{n+1}\ric(\omega) 
$$
where $n$ is the complex dimension of $M$.
Let $\psi$ be a defining function of $\Omega$ which is strictly plurisubharmonic on a neighborhood of $\partial \Omega$. Then $-\log(-\psi)$ is strictly plurisubharmonic near $\partial \Omega$. Computation shows that
$$
\omega^0_{\psi}:=\omega^0-{i\partial\bar\partial}\log(-\psi)
$$
is a complete K\"{a}hler metric on $\Omega$. Moreover, $(\Omega,\omega^0_{\psi})$ has bounded geometry of infinite order (see Proposition 1.3 in \cite{Cheng_Yau}).
In this setting, the following theorem due to Cheng and Yau gives a solution of the complex Monge-Amp\`{e}re equation.

\begin{theorem}[Cheng, Yau \cite{Cheng_Yau}]\label{CY}
If $F\in C^{\infty}(\ov{\Omega})$, then there exists a solution $\phi$ of the equation:
\begin{equation}\label{E:Monge-Ampere}
(\omega^0_{\psi}+{i\partial\bar\partial}\phi)^n=e^{(n+1)\phi+F}(\omega^0_{\psi})^n,
\end{equation}
which is called a complex Monge-Amp\`{e}re equation.
\end{theorem}
Applying Theorem \ref{CY} with
$
F:=\log\Big[(-\psi)^{-(n+1)}\frac{(\omega)^n}{(\omega^0_{\psi})^n}\Big]
\in C^{\infty}(\ov{\Omega})$, the complex Monge-Amp\`{e}re equation implies that
$$
\ric(\omega^0_{\psi}+{i\partial\bar\partial}\phi)=-(n+1)(\omega^0_{\psi}+{i\partial\bar\partial}\phi).
$$
The uniqueness of the \ke metric $\omega^{KE}$ on $\Omega$ says that
$$
\omega^{KE}=\omega^0_{\psi}+{i\partial\bar\partial}\phi.
$$ 

\subsection{Boundary behavior of the solution of the Monge-Amp\`{e}re equation}

Notice that the solution $\phi$ of \eqref{E:Monge-Ampere} depends on $F$, which is determined by the choice of a defining function $\psi$. To obtain a high vanishing order of $\phi$ near the boundary of $\Omega$, Fefferman invented a method to get a new defining function $\widetilde\psi$:

\begin{lemma}[Lemma 3.4 in \cite{Coevering}]\label{L:Definingfunction}
There exists a new defining function $\widetilde\psi=\eta\cdot\psi$ of $\Omega$ with a positive function $\eta\in C^{\infty}(\ov{\Omega})$ such that
$$
F=O(\psi^{n+1}),
$$
where $
F:=\log\Big[(-\widetilde\psi)^{-(n+1)}(\omega)^n/(\omega^0_{\widetilde\psi})^n\Big]
$.
\end{lemma}

Define the new reference metric by
$$
\omega^0_{\widetilde\psi}:=\omega^0-{i\partial\bar\partial}\log(-\widetilde\psi)=\omega^0_{\psi}-{i\partial\bar\partial}\log\eta.
$$ 
We apply Theorem \ref{CY} with $F$ in Lemma \ref{L:Definingfunction} to solve the following new Monge-Amp\`{e}re equation:
\begin{equation} \label{E:Monge-Ampere2}
(\omega^0_{\widetilde\psi}+{i\partial\bar\partial}\phi)^n=e^{(n+1)\phi+F}(\omega^0_{\widetilde\psi})^n
\end{equation}
Then the solution $\phi$ of \eqref{E:Monge-Ampere2} gives the K\"{a}hler-Einstein metric
$$
\omega^0_{\widetilde\psi}+{i\partial\bar\partial}\phi=\omega^{KE}.
$$
In \cite{Cheng_Yau}, Cheng and Yau obtained an asymptotic boundary behavior of the solution $\phi$ of \eqref{E:Monge-Ampere2} as follows:

\begin{theorem}[Theorem 6.5 in \cite{Cheng_Yau}] \label{T:CY2}
Suppose that $\phi$ is a solution of \eqref{E:Monge-Ampere2}. Then
\begin{equation*}
\abs{D^k\phi}=O(\abs{\psi}^{n+1/2-k-\epsilon}),
\end{equation*}
where $\abs{D^k\phi}$ is the Euclidean length of the $k$-th derivative of $\phi$.
\end{theorem}
Theorem \ref{T:CY2} implies that for $\epsilon>0$ and $1\le\alpha,\beta\le{n}$,
\begin{equation*} \label{E:boundary}
\abs{\phi_{\alpha\bar\beta}}\le O(\abs{\psi}^{n-3/2-\epsilon}),
\end{equation*}
where 
$\phi_{\alpha\bar\beta}:=\frac{\partial^2}{\partial z^{\alpha}\partial\ov{z^{\beta}}}\phi$ and
$(z^1,\ldots,z^n)$ are local coordinates near a point in $\partial\Omega$.
Therefore, we have 
$$
\phi_{\alpha\bar\beta}\in{C}^\infty(\Omega)\cap{C}^{n-3/2-\epsilon}(\ov\Omega).
$$
\begin{remark}
In \cite{Lee_Melrose}, Lee and Melrose obtained the asymptotic expansion of the solution of a complex Monge-Amp\`{e}re equation for strongly pseudoconvex domains in $\CC^n$. In particular, $\phi$ satisfies the following {\it optimal} estimate:
\begin{equation*}
\abs{\phi_{\alpha\bar\beta}}\le O(\abs{\psi}^{n-1-\epsilon})
\end{equation*}
for $\epsilon>0$. This is also valid in our situation (see Section 3.4 in \cite{Coevering}).
\end{remark}

\section{Variation of \ke metrics}

In this section, we shall define the variation of \ke metrics $\rho$ for a holomorphic family of strongly pseudoconvex domains. We also explain why the hypothesis of a holomorphic family of strongly pseudoconvex domain is necessary for the regularity of $\rho$ (cf. \cite{Choi1, Paun, Choi_ych}).

\subsection{Holomorphic family of strongly pseudoconvex domains}

Let $p:X\rightarrow Y$ be a surjective holomorphic map, where $X^{n+d}$ and $Y^d$ are complex manifolds and let $D$ be a bounded smooth domain in $X$. For $y\in S:=p(D)$, denote $X_y:=p^{-1}(y)$ and $D_y:=D\cap X_y$.
We denote by $W\subset Y$ the set of all singular values of $p\vert_D$ and $p\vert_{\partial D}$ in $S$.

\begin{definition}\label{family}
A surjective holomorphic map $p:D\rightarrow S$ is called a {\it holomorphic family of bounded strongly pseudoconvex domains} in $X$ (with degenerations) if it satisfies the following:
\begin{itemize}
\item[(1)] $p$ is proper on $\ov{D}$.
\item[(2)] $D_y$ is a strongly pseudoconvex domain in $X_y$ for $y\in S':=S\setminus W$. 
\end{itemize}
In particular, $p:D'\rightarrow S'$ is called a holomorphic family of bounded strongly pseudoconvex domains {\it without degenerations}, where $D':=D\setminus p^{-1}(W)$.
\end{definition}
\begin{remark}\label{Ehresmann}
Condition (1) in Definition \ref{family} implies that every generic fibers are diffeomorphic to each other:
for a fixed point $y_0\in S'$, choose an open neighborhood $V$ of $y_0$ in $S'$. Then $\ov D_V:=\ov D\cap p^{-1}(V)$ is a manifold with boundary $\partial D_V:=\partial D\cap p^{-1}(V)$. The condition implies that $p:\ov D_V\rightarrow V$ is a proper submersion such that the restriction $p:\partial D_V\rightarrow V$ is also a submersion. By Ehresmann's fibration theorem for manifolds with boundaries (cf. Theorem 1.4 in \cite{Saeki}), there exists a diffeomorphism $\Phi:\ov D_V\rightarrow \ov{D_{y_0}}\times V$. Therefore, the fibers $\ov{D_y}$ with $y\in V$ are diffeomorphic to each other.
\end{remark}

\subsection{Regularity of variations of K\"{a}hler-Einstein metrics}\label{regularity}

Fix $y_0\in S'$. Let $\psi$ be a defining function of $\Omega:=D_{y_0}$ which is strictly plurisubharmonic on a neighborhood of $\partial{D_{y_0}}$. Since $D$ is a smooth bounded domain, Condition (2) in Definition \ref{family} implies that there exists a smooth defining function $r$ of $D$ and a neighborhood $V\subset S'$ of $y_0$ satisfying:
\begin{itemize}
\item [(\romannumeral1)] $r_{y_0}:=r|_{D_{y_0}}=\psi$,
\item [(\romannumeral2)] $r_y$ is a defining function of $D_y$, strictly plurisubharmonic on a neighborhood of $\partial D_y$ for each $y\in V$.
\end{itemize}
The last condition implies that $-\log(-r_y)$ is strictly plurisubharmonic.

From now on, suppose that there exists a K\"{a}hler form $\omega_X$ on $X$ satisfying 
$$
\ric(\omega_y)<0
$$
for all $y\in S'$, where $\omega_y:=\omega_X|_{X_y}$.
Then the metric $\omega^0_{r_y}$, defined by
$$
\omega^0_{r_y}:=\omega^0_y-{i\partial\bar\partial}\log(-r_y),
$$
where $\omega^0_y:=-\frac{1}{n+1}\ric(\omega_y)$ is a complete K\"{a}hler metric on $D_y$ for each $y\in S'$. Therefore, Theorem \ref{CY} implies that for each fiber $D_y$, there exists a solution $\varphi_y$ of the Monge-Amp\`{e}re equation:
\begin{equation}\label{E:Monge-Ampere-Variation}
(\omega^0_{r_y}+{i\partial\bar\partial}\varphi_y)^n=e^{(n+1)\varphi_y+F_y}(\omega^0_{r_y})^n,
\end{equation}
with
$$
F_y:=\log[(-r_y)^{-(n+1)}(\omega)^n/(\omega^0_{r_y})^n].
$$
The \ke metric $\omega^{KE}_y$ of $D_y$ can be expressed by
$$
\omega^{KE}_y=\omega^0_{r_y}+{i\partial\bar\partial}\varphi_y
=\omega^0_y+{i\partial\bar\partial}(-\log(-r_y)+\varphi_y).
$$ 
Now the following proposition shows the regularity of variation of \ke metrics in the base direction.

\begin{proposition}\label{smooth}
The function $\varphi$, given by
$$
\varphi(x):=\varphi_y(x)
$$
where $y=p(x)$ is smooth on $D'$
\end{proposition}
\begin{proof}
Recall that fibers $\ov{D_y}$ with $y\in U$ are diffeomorphic to $\ov\Omega$ (Remark \ref{Ehresmann}).
Therefore, we can identify the function space $\tilde{C}^{k+\epsilon}(D_y)$, which is the Banach space with the norm given by covering of bounded geometry (see the definition in \cite{Cheng_Yau}), with $\tilde{C}^{k+\epsilon}(\Omega)$ for all $y\in U$. Now the conclusion follows from the implicit function theorem for Monge-Amp\`{e}re operator (cf. Section 3 in \cite{Choi1}).
\end{proof}


\subsection{Construction of the curvature form}
\label{S:construction}

Since $(X,\omega_X)$ is a K\"ahler manifold, $\omega_X$ induces a singular hermitian metric $h_{X/Y}$ on the relative canonical line bundle $K_{X/Y}$ as follows (cf. \cite{Paun, Choi_ych}):

Let $x\in X$ and $y:=p(x)\in Y$. Given any local coordinate system $(z^1,\ldots,z^{n+d})$ for $x$ on $U\subset X$ and $(s^1,\ldots,s^d)$ for $y$ on the image of $U$ under $p$ in $Y$, denote the corresponding Euclidean volume forms by 
$$
dV_z
:=
i dz^1\wedge d\overline{z^1}\wedge \cdots\wedge i dz^{n+d} \wedge d\ov{z^{n+d}},
$$
$$
dV_s
:=
i ds^1\wedge d\overline{s^1}\wedge \cdots\wedge i ds^d \wedge d\ov{s^d},
$$
Then the singular hermitian metric $h_{X/Y}$ on the relative canonical line bundle $K_{X/Y}$ is defined by the local weight function $\Psi_U$, given by
\begin{equation}\label{hermitian}
e^{-\Psi_U}:=\frac{dV_z}{(\omega_X)^n\wedge p^{\ast}(dV_s)}.
\end{equation}
The corresponding curvature current $\Theta_{h_{X/Y}}$ on $K_{X/Y}$ is given by
\begin{equation}\label{curvature}
\Theta_{h_{X/Y}}:={i\partial\bar\partial}\Psi_U={i\partial\bar\partial}\log((\omega_X)^n\wedge p^{\ast}(dV_s)).
\end{equation}

Let $Z$ be the set of singular values of $p$ in Y. Denote by $Y':=Y\setminus Z$ and $X':=X\setminus p^{-1}(Z)$.
Then the restriction map $p:X'\rightarrow Y'$ is a submersion so that there is no singular point of $p$ in $X'$. Therefore, $h_{X'/Y'}$ is a smooth hermitian metric on $K_{X'/Y'}$ and the curvature $\Theta_{h_{X'/Y'}}$ is a smooth $(1,1)$-form on $X'$. Moreover, the equation \eqref{curvature} implies that
$$
\Theta_{h_{X'/Y'}}|_{D_y}=-\ric(\omega_y)=(n+1)\omega^0_y.
$$
Proposition \ref{smooth} implies that the $d$-closed real $(1,1)$-form $\rho$, given by
$$
\rho:=\frac{1}{n+1}\Theta_{h_{X'/Y'}}+{i\partial\bar\partial}(-\log(-r))+{i\partial\bar\partial}\varphi
$$
is well-defined and smooth on $D'$. 
Then $\rho$ is also called a {\it variation of \ke metrics} (or {\it fiberwise \ke metrics}), because
\begin{align}
\rho|_{D_y}&=\frac{1}{n+1}\Theta_{h_{X'/Y'}}|_{D_y}+{i\partial\bar\partial}(-\log(-r_y)+\varphi_y)\\
&=\omega^0_y+{i\partial\bar\partial}(-\log(-r_y)+\varphi_y)=\omega_y^{KE}.\label{fiberwise}
\end{align}
\begin{remark}
The variation of \ke metrics $\rho$ can be considered as a curvature form on the relative canonical line bundle $K_{D'/S'}$: \eqref{fiberwise} implies that $\rho$ is a smooth form on $D'$, which is positive-definite along each fibers. Therefore, $\rho$ induces another smooth hermitian metric $h_{D'/S'}$ on $K_{D'/S'}$. The Monge-Amp\`{e}re equation \eqref{E:Monge-Ampere-Variation} implies that the corresponding smooth curvature form $\Theta_{h_{D'/S'}}$ on $D'$ can be computed as
$$
\Theta_{h_{D'/S'}}={i\partial\bar\partial}\log(\rho^n\wedge p^{\ast}dV_s)=(n+1)\rho.
$$
\end{remark}


\section{Positivity of the curvature form}

In this section, we will prove Theorem \ref{main}. To obtain the positivity of $\rho$, we use an asymptotic boundary behavior of the geodesic curvature $c(\rho)$.

\subsection{Proof of Theorem \ref{main}}

Note that it is enough to prove the theorem for base spaces of dimension one assuming $S\subset\mathbb{C}$. Let $x$ be a point in $D'$ and $s$ be a holomorphic coordinate, centered at $y:=p(x)\in S'$. With abuse of notation $s:=s\circ p$, let $(z^1,\ldots,z^n,s)$ be a local coordinate system for $x$ on $U\subset X$ such that $(z^1,\ldots,z^n)$ is a local coordinate system in $X_y$. Define a function $h:U\cap D'\rightarrow\mathbb{C}$ by
$$
h:=\frac{1}{n+1}\Psi_U-\log(-r)+\varphi.
$$
Then $\rho$ can be written as $\rho=i\partial\ov\partial h$, i.e.,
$$
\rho=i(h_{s\ov{s}}ds\wedge d\ov{s}+h_{\alpha\ov{s}}dz^{\alpha}\wedge ds+h_{s\ov{\beta}}ds\wedge d\ov{z^{\beta}}+h_{\alpha\ov\beta}dz^{\alpha}\wedge d\ov{z^{\beta}}),
$$
where subscripts denote the differentiation along the corresponding coordinate direction. Note that $\rho$ is positive-definite along each fiber $D_y$ (i.e., the matrix $(h_{\alpha\ov\beta})$ is invertible). Hence it has at least $n$ positive eigenvalues. To show that $\rho$ is positive-definite, we have to show that the $(n+1)$-th eigenvalue (in the ``base direction") is positive. In order to do this, we consider the form $\rho^{n+1}$ on $D'$. It is well-know that $\rho$ satisfies
\begin{equation}\label{fiber_positive} 
\rho^{n+1}=c(\rho)\cdot \rho^n\wedge ids\wedge d\ov{s},
\end{equation}
where the function $c(\rho)$ is a globally defined on $D$. In terms of local coordinates, $c(\rho)$ can be expressed as
\begin{equation} 
c(\rho)=
h_{s\bar{s}}-h_{s\bar\beta}h^{\bar\beta\alpha}h_{\alpha\bar{s}},
\end{equation}
where $(h^{\bar\beta\alpha})$ is the inverse matrix of $(h_{\alpha\ov\beta})$.
The function $c(\rho)$ is called the {\it geodesic curvature} of $\rho$ (for exact definition, see \cite{Schumacher, Choi1}). 
Note that \eqref{fiber_positive} implies that $\rho$ is positive definite on $D'$ if and only if $c(\rho)|_{D_y}>0$ for every $y\in S'$.

To show the positivity of $c(\rho)$ on $D_y$, we will use the following elliptic equation, obtained by Schumacher in \cite{Schumacher}. Denote by $\Delta$ the Laplace-Beltrami operator with respect to the K\"ahler-Einstein metric $\omega^{KE}_y$ of $D_y$.

\begin{proposition} [Schumacher \cite{Schumacher,Choi1}]\label{PDE}
The geodesic curvature $c(\rho)$ satisfies the elliptic equation:
\begin{equation} \label{E:PDE}
-\Delta c(\rho)+(n+1)c(\rho)=\norm{\bar\partial v_{\rho}}^2
\end{equation}
on each fiber $D_y$, where $v_{\rho}$ is the horizontal lift of $v:=\partial/\partial s$ with respect to $\rho$.
\end{proposition}

In the next subsection, we will show that $c(\rho)$ is bounded from below if $D$ is strongly pseudoconvex in $X$. Assuming this fact, we can apply Yau's almost maximum principle. This implies that there exists a sequence $\{x_k\}\subset D_y$ such that
\begin{itemize}
	\item[(i)] $\inf\limits_{x\in D_y} c(\rho)(x)=\lim\limits_{k\rightarrow\infty}c(\rho)(x_k)$,
	\item[(ii)] $\lim\limits_{k\rightarrow\infty}\nabla c(\rho)(x_k)=0$, and $\liminf\limits_{k\rightarrow\infty}\Delta c(\rho)(x_k)\geq0$.
\end{itemize}
It follows from Proposition \ref{PDE} that
$$
(n+1)c(\rho)(x_k)=\norm{\bar\partial v_{\rho}}^2+\Delta c(\rho)(x_k)\geq0.
$$
Taking $k\rightarrow\infty$, we have $c(\rho)\geq0$. Since the \ke metric is real-analytic, $c(\rho)$ and $v_{\rho}$ are also real-analytic. Therefore, we can apply the following proposition.
\begin{proposition} [cf. \cite{Schumacher,Choi1}]
Let $u$ and $f$ be real-analytic, non-negative, real function on a neighborhood $U\subset\mathbb{C}^n$ of $0$. Let $\omega_U$ be a real-analytic K\"{a}hler form on $U$ and $C$ be a positive constant. Suppose
$$
-\Delta_{\omega_U} u+Cu=f
$$
holds. If $u(0)=0$, then both $u$ and $f$ are vanish identically in a neighborhood of $0$.
\end{proposition}
The above proposition with $u:=c(\rho)$ implies that $c(\rho)\equiv0$ or $c(\rho)>0$. Now the conclusion of Theorem \ref{main} follows from the following:

\begin{proposition}\label{blowup}
For each fiber $D_y$,
\begin{equation}\label{E:blowup}
c(\rho)(x)\rightarrow\infty
\;\;\;
\textrm{as}
\;\;\;
x\rightarrow\partial D_y,
\end{equation}
provided $D$ is a strongly pseudoconvex domain.
\end{proposition}
Note that this proposition also implies that $c(\rho)$ is bounded from below on $D_y$, as we required.

\subsection{Boundary behavior of the geodesic curvature}

In this subsection, we will prove Proposition \ref{blowup}.
Recall that we constructed a defining function $r$ of $D$ in Section \ref{regularity}. Then Lemma \ref{L:Definingfunction} implies that for each fiber $D_y$ with $y\in V\subset S'$, there exists a new defining function $\tilde{r}_y=\eta_y\cdot r_y$ such that the solution $\varphi_y$ of the Monge-Amp\`{e}re equation \eqref{E:Monge-Ampere2} satisfies
\begin{equation}\label{fiberwise_regularity}
\abs{(\varphi_y)_{\alpha\bar\beta}}\le O(\abs{r_y}^{n-3/2-\epsilon}),
\end{equation}
in a local coordinate system $(z^1,\ldots,z^n,s)$ on $U$ of a point $x_0\in\partial D_y$.
Furthermore, the proof tells us that the function $\eta$, defined by
$$
\eta(x):=\eta_y(x)
$$
where $y=p(x)\in V$ is positive smooth on $\ov{D}_V:=\ov{D}\cap p^{-1}(V)$. Therefore, without loss of generality, we may assume that there exists a defining function $r$ of $D$ such that the solution $\varphi_y$ of the Monge-Amp\`{e}re equation \eqref{E:Monge-Ampere2} with $r_y$ satisfies the estimate \eqref{fiberwise_regularity} (here, we use abuse of notation $r$ for the function $\tilde{r}:=\eta\cdot r$).

Define a $(1,1)$-form on $\ov{D}_V$ by
$$
\tau_r:=i\partial\ov\partial g_r,
$$
where $g_r:=-\log(-r)$.
In terms of local coordinates, we have
$$
\tau_r=i((g_r)_{s\ov{s}}ds\wedge d\ov{s}+(g_r)_{\alpha\ov{s}}dz^{\alpha}\wedge ds+(g_r)_{s\ov{\beta}}ds\wedge d\ov{z^{\beta}}+(g_r)_{\alpha\ov\beta}dz^{\alpha}\wedge d\ov{z^{\beta}}).
$$
Since $\tau_r$ is positive-definite along each fiber $D_y$ for $y\in V$, we can consider the geodesic curvature $c(\tau_r)$ of $\tau_r$. Direct computation yields the following:

\begin{proposition}
For each fiber $D_y$,
\begin{equation*}
c(\tau_r)(x)\rightarrow\infty
\;\;\;
\textrm{as}
\;\;\;
x\rightarrow\partial D_y,
\end{equation*}
provided $D$ is a strongly pseudoconvex domain.
\end{proposition}
\begin{proof}
See Remark 2 in Section 5.1 in \cite{Choi1}.
\end{proof}

Now the proof of Proposition \ref{blowup} is complete by the following proposition. This yields that the geodesic curvatures $c(\tau_r)$ and $c(\rho)$ go to infinity near the boundary of the same order.

\begin{proposition}\label{P:ratio_boundary_behavior}
\begin{equation}\label{E:boundedness}
\frac{c(\rho)}{c(\tau_r)}(x)\rightarrow1
\;\;\;
\textrm{as}
\;\;\;
x\rightarrow\partial D_y.
\end{equation}
\end{proposition}

\begin{proof}
Fix a point $x_0\in\partial D_y$. Let $U$ be a neighborhood of $x_0$ in $X'$ and $(z^1,\ldots,z^n,s)$ be coordinates on $U$. Denote by $\tau^0$ the smooth $(1,1)$-form on $X'$ which is given by
$$
\tau^0:=\frac{1}{n+1}\Theta_{h_{X'/Y'}}.
$$
Define a function $g^0:=\frac{1}{n+1}\Psi_U$. Then $\tau^0$ can be written as $\tau^0=i\partial\ov\partial g^0$ on $U$, i.e.,
$$
\tau^0=i((g^0)_{s\ov{s}}ds\wedge d\ov{s}+(g^0)_{\alpha\ov{s}}dz^{\alpha}\wedge ds+(g^0)_{s\ov{\beta}}ds\wedge d\ov{z^{\beta}}+(g^0)_{\alpha\ov\beta}dz^{\alpha}\wedge d\ov{z^{\beta}}).
$$
We also define a smooth $(1,1)$-form $\tau^0_r$ on $\ov{D}_V$ by
$$
\tau^0_r:=\tau^0+\tau_r.
$$
Then $\tau^0_r$ can be expressed as $\tau^0_r=i\partial\ov\partial g^0_r$, where $g^0_{r}:=g^0+g_r$ is a function on $U\cap \ov{D}_V$. 
In terms of local coordinates, we have
$$
\tau^0_r=i((g^0_r)_{s\ov{s}}ds\wedge d\ov{s}+(g^0_r)_{\alpha\ov{s}}dz^{\alpha}\wedge ds+(g^0_r)_{s\ov{\beta}}ds\wedge d\ov{z^{\beta}}+(g^0_r)_{\alpha\ov\beta}dz^{\alpha}\wedge d\ov{z^{\beta}}).
$$

Note that
$$
h:=\frac{1}{n+1}\Psi_U-\log(-r)+\varphi=g^0_r+\varphi.
$$
Then the equation $(5.3)$ in \cite{Choi1} implies that we have following expression:
$$
h^{\bar\beta\alpha}-(g^0_{{r}})^{\bar\beta\alpha}=(g^0_{{r}})^{\bar\beta\gamma}N_{\gamma\bar\delta}(g^0_{{r}})^{\bar\delta\alpha},
$$
where $N=(N_{\alpha\bar\beta})$ is a hermitian $n\times{n}$ matrix.
Now we can compute $c(\rho)$ as follows:
\begin{align*}
c(\rho)
&=
h_{s\bar{s}}-h_{s\bar\beta}h^{\bar\beta\alpha}h_{\alpha\bar{s}} \\	
&=
h_{s\bar{s}}-h_{s\bar\beta}
\paren{(g^0_{{r}})^{\bar\beta\alpha}
+
(g^0_{{r}})^{\bar\beta\gamma}N_{\gamma\bar\delta}(g^0_{{r}})^{\bar\delta\alpha}
}
h_{\alpha\bar{s}} \\
&=
(g^0_{{r}})_{s\bar{s}}+\varphi_{s\bar{s}}
-
\paren{(g^0_{{r}})_{s\bar\beta}+\varphi_{s\bar\beta}}\paren{(g^0_{{r}})^{\bar\beta\alpha}
+
(g^0_{{r}})^{\bar\beta\gamma}N_{\gamma\bar\delta}(g^0_{{r}})^{\bar\delta\alpha}
}
\paren{(g^0_{{r}})_{\alpha\bar{s}}+\varphi_{\alpha\bar{s}}} \\
&=
c(\tau^0_{{r}}) + R_1 + R_2,
\end{align*}
where the remaining terms $R_1$ and $R_2$ are defined by
\begin{align*}
R_1:&=
	\varphi_{s\bar{s}}+\paren{(g^0_{{r}})_{s\bar\beta}(g^0_{{r}})^{\bar\beta\alpha}\varphi_{\alpha\bar{s}}
	+\varphi_{s\bar\beta}(g^0_{{r}})^{\bar\beta\alpha}(g^0_{{r}})_{\alpha\bar{s}}
	+\varphi_{s\bar\beta}(g^0_{{r}})^{\bar\beta\alpha}\varphi_{\alpha\bar{s}}
	},\\
R_2
	:&=\paren{(g^0_{{r}})_{s\bar\beta}+\varphi_{s\bar\beta}}
	(g^0_{{r}})^{\bar\beta\gamma}N_{\gamma\bar\delta}(g^0_{{r}})^{\bar\delta\alpha}
	\paren{(g^0_{{r}})_{\alpha\bar{s}}+\varphi_{\alpha\bar{s}}}.
\end{align*}
Again the equation $(5.3)$ in \cite{Choi1} implies that
$$
(g^0_{{r}})^{\bar\beta\alpha}-(g_{{r}})^{\bar\beta\alpha}=(g_{{r}})^{\bar\beta\gamma}M_{\gamma\bar\delta}(g_{{r}})^{\bar\delta\alpha},
$$
where $M=(M_{\alpha\bar\beta})$ is a hermitian $n\times{n}$ matrix.
Then, we have
\begin{align*}
c(\tau^0_{{r}})
&=
(g^0_{{r}})_{s\bar{s}}-(g^0_{{r}})_{s\bar\beta}(g^0_{{r}})^{\bar\beta\alpha}(g^0_{{r}})_{\alpha\bar{s}} \\	
&=
(g^0_{{r}})_{s\bar{s}}-(g^0_{{r}})_{s\bar\beta}
\paren{(g_{{r}})^{\bar\beta\alpha}
+
(g_{{r}})^{\bar\beta\gamma}M_{\gamma\bar\delta}(g_{{r}})^{\bar\delta\alpha}
}
(g^0_{{r}})_{\alpha\bar{s}} \\
&=
(g^0)_{s\bar{s}}+(g_{r})_{s\bar{s}}
-
\paren{(g^0)_{s\bar\beta}+(g_{r})_{s\bar\beta}}\paren{(g_{{r}})^{\bar\beta\alpha}
+
(g_{{r}})^{\bar\beta\gamma}M_{\gamma\bar\delta}(g_{{r}})^{\bar\delta\alpha}
}
\paren{(g^0)_{\alpha\bar{s}}+(g_{r})_{\alpha\bar{s}}} \\
&=
c(\tau_{{r}}) + R_3 + R_4,
\end{align*}
where the remaining terms $R_3$ and $R_4$ are given by
\begin{align*}
R_3&:=
	(g^0)_{s\bar{s}}+\paren{(g_{r})_{s\bar\beta}(g_{{r}})^{\bar\beta\alpha}(g^0)_{\alpha\bar{s}}
	+(g^0)_{s\bar\beta}(g_{{r}})^{\bar\beta\alpha}(g_{r})_{\alpha\bar{s}}
	+(g^0)_{s\bar\beta}(g_{{r}})^{\bar\beta\alpha}(g^0)_{\alpha\bar{s}}
	},\\
R_4
	&:=\paren{(g_{r})_{s\bar\beta}+(g^0)_{s\bar\beta}}
	(g_{{r}})^{\bar\beta\gamma}M_{\gamma\bar\delta}(g_{{r}})^{\bar\delta\alpha}
	\paren{(g_{r})_{\alpha\bar{s}}+(g^0)_{\alpha\bar{s}}}.
\end{align*}

Now we have
$$
\frac{c(\rho)}{c(\tau_r)}=1+\frac{R_1+R_2}{c(\tau_r)}+\frac{R_3+R_4}{c(\tau_r)}.
$$
Since $g^0$ is smooth on $U\cap \ov{D}_V$, all derivatives are smooth and bounded in $U\cap \ov{D}_V$. It is easy to see that $(g_{{r}})^{\bar\beta\alpha}$ is also smooth on $U\cap \ov{D}_V$ and $(g_{{r}})^{\bar\beta\alpha}=O(\abs{r})$. This implies that $(g_{r})_{s\bar\beta}(g_{{r}})^{\bar\beta\alpha}$ and $(g_{{r}})^{\bar\beta\alpha}(g_{r})_{\alpha\bar{s}}$ are bounded (cf. Corollary 5.5 in \cite{Choi1}). Therefore, the remaining term $R_3$ is bounded. The term $R_4$ is also bounded by the following lemma.

\begin{lemma} \label{L:identity}
For any $\alpha,\beta\in \{1,\ldots,n\}$, we have
$$
M_{\alpha\bar\beta}\in C^\infty(U\cap \ov{D}_V) 
$$
so that $(g^0_{{r}})^{\bar\beta\alpha}\in{C}^\infty(U\cap \ov{D}_V)$. In particular, $(g^0_{{r}})^{\bar\beta\alpha}=O(\abs{r})$.
\end{lemma}

\begin{proof}
The proof is essentially same as the proof of Lemma 5.3 in \cite{Choi1}.
\end{proof}
Hence, it is enough to show that 
\begin{equation*}
\frac{R_1+R_2}{c(\tau_{{r}})}(x)\rightarrow0
\;\;\;
\textrm{as}
\;\;\;
x\rightarrow x_0\in\partial D_y.
\end{equation*}
As in the estimate of $R_3$, Lemma \ref{L:identity} implies that $(g^0_{r})_{s\bar\beta}(g^0_{{r}})^{\bar\beta\alpha}$ and $(g^0_{{r}})^{\bar\beta\alpha}(g^0_{r})_{\alpha\bar{s}}$ are bounded on $U\cap \ov{D}_V$ (cf. Corollary 5.5 in \cite{Choi1}). 
The proof of Proposition \ref{smooth} implies that $\varphi_{s\ov{s}}$ is bounded.
Recall that $\varphi$ satisfies \eqref{fiberwise_regularity}. Applying the Schauder estimates to $\varphi_s$ and $\varphi_{\ov{s}}$, we can show that
$$
|\varphi_{\alpha\ov{s}}|=O(|r|^{-1/2-\epsilon}) \text{\ \ and\ \ }  |\varphi_{s\ov{\beta}}|=O(|r|^{-1/2-\epsilon}),
$$
for some $\epsilon$ with $0<\epsilon<1/2$ (for detailed proof, see Section 3.3 in \cite{Choi2}). Then above estimates imply that
\begin{equation*}
\frac{R_1}{c(\tau_{{r}})}(x)\rightarrow0
\;\;\;
\textrm{as}
\;\;\;
x\rightarrow x_0\in\partial D_y.
\end{equation*}
and that $R_2$ is bounded by the following lemma.

\begin{lemma} [cf. Lemma 5.4 in \cite{Choi1}]\label{L:identity2}
Let $N^y:=(N^y_{\alpha\bar\beta})$ be the $n\times n$ matrix with $N^y_{\alpha\bar\beta}:=N_{\alpha\bar\beta}|_{D_y}$.
For any $\alpha,\beta\in \{1,\ldots,n\}$, we have
$$
N^y_{\alpha\bar\beta}\in C^\infty(U\cap D_y)\cap C^{n-3/2-\epsilon}(U\cap \ov{D}_y),
$$
with $||N^y||=O(|r_y|^{n-3/2-\epsilon})$.
\end{lemma}
This completes the proof of Proposition \ref{P:ratio_boundary_behavior}.
\end{proof}


\section{Extension of the curvature form}

In this section, we will prove Theorem \ref{extension}, using the argument of P\u{a}un in \cite{Paun}. Main tools are Demailly's approximation theorem (Theorem \ref{T:Demailly}) and the Ohsawa-Takegoshi theorem (Theorem \ref{T:BP}). However, there is no uniform boundedness of the volumes of fibers since the fibers are noncompact. To resolve this difficulty, we need another complete metric $\widetilde{\omega}_D$ for applying the Schwarz lemma (Theorem \ref{T:Schwarz}).

Fix an arbitrary point $x_0\in D$ with $y_0:=p(x_0)\in W$. Take a coordinate neighborhood $U$ of $x_0$ in $D$.
\begin{align*}
	\Theta_{h_{D'/S'}}=(n+1)\rho&=\Theta_{h_{X'/Y'}}+(n+1){i\partial\bar\partial}(-\log(-r))+(n+1){i\partial\bar\partial}\varphi\\
	&={i\partial\bar\partial}(\Psi_U-(n+1)\log(-r)+(n+1)\varphi),
\end{align*}
where $\Psi_U$ is a local weight function on $U$, defined in \eqref{hermitian}.
Define a function $\theta:U\cap D'\rightarrow \mathbb{R}$ by
$$
\theta:=(n+1)h=\Psi_U-(n+1)\log(-r)+(n+1)\varphi,
$$
which is a local potential function of $\Theta_{h_{D'/S'}}$. Since $\theta$ is strictly plurisubharmonic by Theorem \ref{main}, it is enough to show that there exists a neighborhood $\widetilde{U}\subset\subset U$ of $x_0$ such that $\theta$ is bounded from above. More precisely, for any $x\in \widetilde{U}\cap D'$, we want to show that there exists a constant $C$ independent of $y:=p(x)$ satisfying
$$
\theta_y:=\theta|_{U_y}=\Psi_U|_{U_y}-(n+1)\log(-r_y)+(n+1)\varphi_y <C
$$
on $\widetilde{U}_y$.

\begin{remark}
Yau's $C^0$-estimate implies that $(n+1)\sup_{D_y}\varphi_y<\sup_{D_y}F_y$. Although $F_y$ is bounded from above on $\ov{D_y}$, we do not know whether there exists a constant $C$, which does not depend on $y$ satisfying $F_y<C$.
\end{remark}

First, we approximate $\theta_y$ by the logarithm of absolute values of holomorphic functions, using the following theorem.

\begin{theorem}[Demailly]\label{T:Demailly}
Let $\mathcal{H}^m_y$ be the Hilbert space defined as follows
$$
\mathcal{H}^m_y:=\{f\in \mathcal{O}(U_y): ||f||^2_{m,y}:=\int_{U_y}|f|^2e^{-m\theta_y}(\omega^{KE}_y)^n<\infty\}
$$
Then for every $x\in U_y$, we have
$$
\theta_y(x)=\lim_{m\rightarrow\infty}\sup\frac{1}{m}\log|f(x)|^2,
$$
where the supremum is taken over all $f\in \mathcal{H}^m_y$ satisfying $||f||^2_{m,y}\leq1$.
\end{theorem}

Note that for all $f\in \mathcal{H}^m_y$ satisfying $||f||^2_{m,y}\leq1$, we have
\begin{align*}
\int_{U_y}|f|^{2/m}e^{-\theta_y}(\omega^{KE}_y)^n&\leq\Big(\int_{U_y}|f|^2e^{-m\theta_y}(\omega^{KE}_y)^n\Big)^{1/m}\Big(\int_{U_y}(\omega^{KE}_y)^n\Big)^{\frac{m-1}{m}}\\
&\leq \Big(\int_{U_y}(\omega^{KE}_y)^n\Big)^{\frac{m-1}{m}}=:\Big(Vol_{KE}(U_y)\Big)^{\frac{m-1}{m}}.
\end{align*}
On the other hand, the Monge-Amp\`{e}re equation (\ref{E:Monge-Ampere}) implies that
$$
e^{-\theta_y}(\omega^{KE}_y)^n=e^{-\Psi_{U}}e^{-(n+1)(-\log(-r_y)+\varphi_y )}(\omega^{KE}_y)^n
=\frac{dV_z}{(\omega_X)^n\wedge p^{\ast}dV_s}(\omega_y)^n=\frac{dV_z}{p^{\ast}dV_s}.
$$
Therefore, we have
$$
\int_{U_y}|f|^{2/m}\frac{dV_z}{p^{\ast}dV_s}=\int_{U_y}|f|^{2/m}e^{-\theta_y}(\omega^{KE}_y)^n\leq\Big(Vol_{KE}(U_y)\Big)^{\frac{m-1}{m}}.
$$
Now we apply the following $L^{2/m}$ version of Ohsawa-Takegoshi theorem:

\begin{theorem} [Berndtsson, P\u{a}un  \cite{Berndtsson_Paun}]\label{T:BP}
There exists a holomorphic function $\tilde{f}$ on $U$ and positive constant $C_0$ independent of $m$ and $y$ satisfying $\tilde{f}|_{U_y}=f$ on $U_y$ and
$$
\int_U|\tilde{f}|^{2/m}dV_z\leq C_0\int_{U_y}|f|^{2/m}\frac{dV_z}{p^{\ast}dV_s}.
$$
\end{theorem}

Choose $\widetilde{U}\subset\subset U$ so that the geodesic ball $B_{\epsilon}(x)$ of radius $\epsilon$ satisfying $B_{\epsilon}(x)\subset U$ for all $x\in V$. Then mean value inequality implies that
$$
|f(x)|^{2/m}=|\tilde{f}(x)|^{2/m}\leq C_{\epsilon}\int_{B_{\epsilon}(x)}|\tilde{f}|^{2/m}dV_z\leq C \Big(Vol_{KE}(U_y)\Big)^{\frac{m-1}{m}}.
$$
Recall that
$$
\theta_y(x)=\lim_{m\rightarrow\infty}\sup\frac{1}{m}\log|f(x)|^2=\lim_{m\rightarrow\infty}\sup\log|f(x)|^{2/m}.
$$
Therefore, to complete the proof of Theorem \ref{extension}, it is enough to show that there exists a positive constant $C$ independent of $y$ satisfying 
$$
Vol_{KE}(U_y):=\int_{U_y}(\omega^{KE}_y)^n<C.
$$

Note that $D$ admits a complete K\"{a}hler metric $\widetilde{\omega}_D$ satisfying $Scal(\widetilde{\omega}_y)>C$ with $\widetilde{\omega}_y:=\widetilde{\omega}_D|_{D_y}$. We apply the following Schwarz lemma for volume forms.

\begin{theorem}[Mok, Yau \cite{Mok_Yau}]\label{T:Schwarz}
Let $(M,h)$ be a complete Hermitian manifold with $Scal(h)\geq-K_1$, and let $N$ be a complex manifold of the same dimension with a volume form $V$ such that the Ricci form is negative definite and
$$
(\frac{i}{2}\partial\ov{\partial}\log V)^n\geq K_2 V.
$$
Suppose $f:M\rightarrow N$ is a holomorphic map and the Jacobian is nonvanishing at
one point. Then $K_1>0$ and
$$
f^{\ast}V\leq\frac{K_1^n}{n^nK_2}(\omega_h)^n,
$$
where $\omega_h$ is the associate $(1,1)$-form of $h$.
\end{theorem}

Let $M=N=D_y$, $f=id$, $\omega_h=\widetilde{\omega}_y$, and $V:=(\omega^{KE}_y)^n$. Applying the above theorem for each fiber $D_y$, we have
$$
(\omega^{KE}_y)^n\leq C(\widetilde{\omega}_y)^n
$$
on $D_y$. 
\begin{remark}
Since we only need boundedness of volume forms locally, we can replace the condition in Theorem \ref{extension} as follows: suppose that for each point $x\in D\setminus D'$, there exists a neighborhood $U(x)$ of $x$, biholomorphic to the unit ball. Assume that the Poincar\'{e} metric $\widetilde{\omega}_P$ on $U(x)$ satisfies $Scal(\widetilde{\omega}_y)>C$, where $\widetilde{\omega}_y:=\widetilde{\omega}_P|_{U_y}$ and $U_y:=U(x)\cap p^{-1}(y)$. Then the above theorem implies that $(\omega^{KE}_y)^n\leq C(\widetilde{\omega}_y)^n$ on $U_y$.
\end{remark}

Now we only need to show that
$$
\int_{U_y}(\widetilde{\omega}_y)^n<C.
$$
The proof is completed by the following theorem of Diederich and Pinchuk.

\begin{theorem} [cf. Theorem 1.4 in \cite{Diederich_Pinchuk}]
Let $p:U\rightarrow V$ be a holomorphic surjective map between $U\subset\mathbb{C}^{n+d}$ and $V\subset\mathbb{C}^d$. Take any $\widetilde{U}\subset\subset U$ with $\widetilde{V}:=p(\widetilde{U})$. Then, there exists a uniform constant $C>0$ such that for all regular value $y\in \widetilde{V}$ of $p$,
$$
Vol(\widetilde{U}_y)\leq C,
$$
where $\widetilde{U}_y:=\widetilde{U}\cap p^{-1}(y)$ and $Vol$ is the $2n$-dimensional Hausdorff measure.
\end{theorem}
\begin{remark}
It is well-known that the $2n$-dimensional Hausdorff measure coincides with the (Riemannian) volume of a $2n$-dimensional submanifolds with respect to the Euclidean metric. 
\end{remark}


\end{document}